\documentclass[11pt]{amsart}

\usepackage{geometry}
\usepackage{palatino}
\usepackage[all]{xy}
\usepackage{amsthm,amssymb,amsmath}
\usepackage{mathrsfs,amscd}
\usepackage{multirow}
\newtheorem{theorem}{Theorem}[section]
\newtheorem{conjecture}[theorem]{Conjecture}

\newtheorem{lemma}[theorem]{Lemma}
\newtheorem{proposition}[theorem]{Proposition}

\newtheorem{remark}[theorem]{Remark}

\newcommand{\bb}[1]{\mathbb{#1}}
\newcommand{\mc}[1]{\mathcal{#1}}
\newcommand{\mf}[1]{\mathfrak{#1}}
\begin{document}

\title{On Quantization of a Nilpotent Orbit Closure in $G_2$}
\author{Kayue Daniel Wong}
\address{Department of Mathematics, Hong Kong University of Science and
Technology, Clear Water Bay, Kowloon, Hong Kong}
\email{makywong@ust.hk}

\begin{abstract}
Let $G$ be the complex exceptional Lie group of type $G_2$. Among the five nilpotent orbits in its Lie algebra $\mf{g}$, only the 8-dimensional orbit $\mc{O}_8$ has non-normal orbit closure $\overline{\mc{O}_8}$. In this manuscript, we will give a quantization model of $\overline{\mc{O}_8}$, verifying a conjecture of Vogan in 1984.
\end{abstract}
\maketitle
\section{Introduction}
Let $G$ be a complex simple Lie group. The $G$-conjugates of a nilpotent element $X \in \mf{g}$ form a \textbf{nilpotent orbit} $\mc{O} \subset \mf{g}$. Following the ideas in \cite{V1} or \cite{V3}, one would like to attach unitary representations to all such orbits along with their finite $G$-equivariant covers. More precisely, let $V$ be a finite $G$-equivariant cover of an affine Poisson $G$-variety containing a nilpotent orbit $\mc{O}$ as an open set, with its ring of regular functions $R(V)$, then one would like to find a (hopefully unitarizable) $(\mf{g}_{\bb{C}}, K_{\bb{C}})$-module $X_{V}$ such that we have the $G$-module isomorphism
$$X_{V}|_{K_{\bb{C}}} \cong R(V).$$
(note that $K \leq G$ is the maximal compact subgroup of $G$, hence its complexification $K_{\bb{C}}$ is isomorphic to $G$). In the following work, we will call $X_{V}$ as a \textbf{quantization} of $V$.\\
\indent As hinted in \cite{V1}, one needs to take special attention when the orbit closure $\overline{\mc{O}}$ is not normal. One reason is due to the algebro-geometric fact that $R(\overline{\mc{O}}) \cong R(\mc{O})$ if and only if  $\overline{\mc{O}}$ is normal. Following the spirit of the orbit method, one needs to give a quantization model for $V = \mc{O}$ and $V = \overline{\mc{O}}$ separately when $\overline{\mc{O}}$ is not normal.\\
\indent Here is a summary on the current progress of the above quantization scheme. In \cite{B 2008}, Barbasch constructs such models for a large class of classical nilpotent orbits. Using a completely different method in \cite{Br 2003}, Ranee Brylinski constructs a Dixmier algebra for all classical nilpotent orbit closures. The reconciliation between the two models is the main theme of the Ph.D. thesis of the author \cite{W2}.\\
\indent Contrary to the classical setting, very little is known about the scheme for exceptional groups. We now focus on the case for $G = G_2$. Write $\{\alpha, \beta \}$ be the simple roots of $\mf{g}$, with $\alpha$ being the short root. The fundamental weights of $\mf{g}$ are therefore given by
$$\{ \omega_1, \omega_2 \} = \{2\alpha + \beta, 3\alpha + 2\beta\}.$$
By the Bala-Carter classification, we have five nilpotent orbits $\mc{O}_0$, $\mc{O}_6$, $\mc{O}_8$, $\mc{O}_{10}$ and $\mc{O}_{12}$ in $\mf{g}$. Following the study of completely prime primitive ideals of Joseph in \cite{J1}, \cite{J2}, Vogan in \cite{V1} conjectured a quantization model for $\mc{O}_8$ and $\overline{\mc{O}_8}$ for $G_2$:
\begin{conjecture}[\cite{V1}, Conjecture 5.6]
Let $\lambda \in \mf{h}^*$ and $J(\lambda)$ be the maximal primitive ideal in $U(\mf{g})$ with infinitesimal character $\lambda$. Then the $(\mf{g}_{\bb{C}}, K_{\bb{C}}) \cong (\mf{g} \times \mf{g}, G)$-modules
$$U(\mf{g})/J(\frac{1}{2}(\omega_1 + \omega_2)),\ \ \ U(\mf{g})/J(\frac{1}{2}(5\omega_1 - \omega_2))$$
are quantizations of $\mc{O}_8$ and $\overline{\mc{O}_8}$ respectively. In particular,
$$U(\mf{g})/J(\frac{1}{2}(\omega_1 + \omega_2))|_{K_{\bb{C}}} \cong R(\mc{O}_8),\ \ \ U(\mf{g})/J(\frac{1}{2}(5\omega_1 - \omega_2))|_{K_{\bb{C}}} \cong R(\overline{\mc{O}_8}).$$
As a consequence, $\mc{O}_8$ has non-normal closure.
\end{conjecture}
Interestingly, by the classification of spherical unitary dual of complex $G_2$ given by Duflo in \cite{Du}, $U(\mathfrak{g})/J(\frac{1}{2}(\omega_1 + \omega_2))$ is unitarizable while $U(\mathfrak{g})/J(\frac{1}{2}(5\omega_1 - \omega_2))$ is not (this fact is also observed by Vogan in p.226 of \cite{V2}). Later, Levasseur and Smith in \cite{LS} proved that $U(\mathfrak{g})/J(\frac{1}{2}(\omega_1 + \omega_2))|_{K_{\mathbb{C}}} \cong R(\mathcal{O}_8)$ and $\overline{\mathcal{O}_8}$ is not normal, but were unable to prove the rest of the conjecture. The main result of this manuscript is the following:
\begin{theorem} \label{thm:main}
As $K_{\bb{C}} \cong G$ modules,
$$U(\mf{g})/J(\frac{1}{2}(5\omega_1 - \omega_2))|_{K_{\bb{C}}} \cong R(\overline{\mc{O}_8}).$$
\end{theorem}
\begin{remark}
This quantization model of nilpotent orbit closure is very different from the classical model given in \cite{Br 2003}. Namely, the Brylinski model is not necessarily of the form $U(\mf{g})/J(\lambda)$. In particular, when the classical nilpotent orbit closure $\overline{\mc{O}}$ is not normal (the classification of all such orbit closures is given in \cite{KP 1982}), the infinitesimal character of the Brylinski model $\lambda_{\mc{O}}$ always yields associated variety $AV(U(\mf{g})/J(\lambda_{\mc{O}})) = \overline{\mc{O}'}$, where $\mc{O}'$ is strictly smaller than $\mc{O}$.\\
\indent In fact, it can be shown that the Brylinski model always contains the composition factor $U(\mf{g})/$ $J(\lambda_{\mc{O}})$. This is part of the on-going work of Barbasch and the author \cite{BW}.
\end{remark}
Before going to the proof of Theorem \ref{thm:main}, it is worthwhile to mention the orbits other than $\mc{O}_8$ in $\mf{g}$. Indeed, Kraft in \cite{K} confirmed that $\mc{O}_8$ is the only nilpotent orbit with non-normal closure. So we just need to consider quantizations of the orbits (and their covers) only. For the zero orbit $\mc{O}_0$ the quantization is trivial, and the quantization of the minimal orbit $\mc{O}_{6}$ is $U(\mf{g})/J(\frac{1}{3}(3\omega_1 + \omega_2))$, where $J(\frac{1}{3}(3\omega_1 + \omega_2))$ is the Joseph ideal. The 10-dimensional orbit $\mc{O}_{10}$ is a special orbit with fundamental group $S_3$. It is a simple exercise to compare the formulas in \cite{BV 1985} and \cite{MG1} that the spherical unipotent representation attached to $\mc{O}_{10}$ is a quantization of $\mc{O}$ (as a bonus, the other two unipotent representations attached to $\mc{O}_{10}$ essentially gives quantization of all covers of $\mc{O}_{10}$ as well). Finally, the quantization of the principal orbit $\mc{O}_{12}$ is well known to be the principal series representation with zero infinitesimal character. In conclusion, we completed the picture of quantization for all nilpotent orbits of $\mf{g}$ and their closures.

\section{Proof of the Theorem}
As mentioned in the Introduction, the non-normality of $\overline{\mc{O}_8}$ implies that $R(\overline{\mc{O}_8}) \subsetneq R(\mc{O}_8)$. In fact, Costantini in \cite{Cos} gives the discrepancies in terms of $G$-modules:
\begin{theorem}[\cite{Cos}, Theorem 5.6] \label{thm:discrepancy}
Let $V_{(a,b)}$ be the finite-dimensional irreducible representation of $G_2$ with highest weight $a\omega_1 + b\omega_2$, where $a$ and $b$ are non-negative integers. Then
$$R(\mc{O}_8) \cong R(\overline{\mc{O}_8}) \oplus \bigoplus_{n \in \bb{N} \cup \{0\}} V_{(1,n)}.$$
\end{theorem}

The following Lemma gives another expression of the discrepancies between $R(\mc{O})$ and $R(\overline{\mc{O}})$:
\begin{lemma} \label{lem:cos}
As virtual $G$-modules,
$$\bigoplus_{n \in \bb{N} \cup \{0\}} V_{(1,n)} = Ind_T^G(1,0) - Ind_T^G(0,1) - Ind_T^G(2,0) + Ind_T^G(1,1) + Ind_T^G(0,2) - Ind_T^G(2,1),$$
where $Ind_T^G(a,b)$ is the shorthand for the induced module $Ind_T^G(e^{a \omega_1 + b\omega_2})$.
\end{lemma}
\begin{proof}
The Lemma can be derived from the Weyl character formula. Namely, by the $W(G_2)$-symmetry of weights of $V_{(a,b)}$, we have
$$V_{(a,b)} = \sum_{w \in W(G_2)} \det(w) Ind_T^G(\lambda_w)$$
with $\lambda_w$ being the unique $W(G_2)$-conjugate of $w[(a,b)+(1,1)] - (1,1)$ lying in the dominant chamber. In fact, we have
\begin{align*}
V_{(1,n)} =  &Ind_T^G(1,n) - Ind_T^G(1,n+3) - Ind_T^G(2,n) + Ind_T^G(2,n+2)\\
& - Ind_T^G(3,n-1) + Ind_T^G(3,n) - Ind_T^G(3,n+1) + Ind_T^G(3,n+2) \\
& + Ind_T^G(6,n-2) - Ind_T^G(6,n) - Ind_T^G(7,n-2) + Ind_T^G(7,n-1)
\end{align*}
for $n > 1$, and
\begin{align*}
V_{(1,0)} &= Ind_T^G(1,0) - Ind_T^G(1,3) - Ind_T^G(2,0) + Ind_T^G(2,2) - Ind_T^G(0,1) + Ind_T^G(3,0)\\
&- Ind_T^G(3,1) + Ind_T^G(3,2) + Ind_T^G(0,2) - Ind_T^G(6,0) - Ind_T^G(1,2) + Ind_T^G(4,1);
\end{align*}
\begin{align*}
V_{(1,1)} &= Ind_T^G(1,1) - Ind_T^G(1,4) - Ind_T^G(2,1) + Ind_T^G(2,3) - Ind_T^G(3,0) + Ind_T^G(3,1) \\
&- Ind_T^G(3,2) + Ind_T^G(3,3) + Ind_T^G(3,1) - Ind_T^G(6,1) - Ind_T^G(4,1) + Ind_T^G(7,0).
\end{align*}
The Lemma is proved by adding up the terms.
\end{proof}
We now study the two Harish-Chandra bi-modules $U(\mf{g})/J(\frac{1}{2}(\omega_1 + \omega_2))$ and $U(\mf{g})/J(\frac{1}{2}(5\omega_1 - \omega_2))$:
\begin{proposition} \label{prop:uni}
As $K_{\bb{C}} \cong G$-modules
$$U(\mf{g})/J(\frac{1}{2}(\omega_1 + \omega_2))|_{K_{\bb{C}}} = Ind_T^G(0,0) - Ind_T^G(0,1) - Ind_T^G(2,0) + Ind_T^G(1,1);$$
$$U(\mf{g})/J(\frac{1}{2}(5\omega_1 - \omega_2))|_{K_{\bb{C}}} = Ind_T^G(0,0) - Ind_T^G(1,0) - Ind_T^G(0,2) + Ind_T^G(2,1).$$
\end{proposition}
\begin{proof}
To cater for subsequent calculations, we let $\mf{h}^* = \{ (x,y,z) \in \bb{C}^3 | x+y+z = 0\}$, with short simple root $\alpha = (1,-1,0)$ and long simple root $\beta = (-1,2,-1)$. Then
$$\lambda_1 = \frac{1}{2}(\omega_1 + \omega_2) = (1,1/2,-3/2);\ \ \lambda_2 = \frac{1}{2}(5\omega_1 - \omega_2) = (2,-1/2,-3/2).$$

The character formulas of $U(\mf{g})/J(\lambda)$ for regular $\lambda$ are well known by the work of Barbasch and Vogan \cite{BV 1985}: Consider the subgroup $W_{\lambda}$ of $W(G_2)$ generated by roots $\alpha$ satisfying $2\frac{\langle \alpha, \lambda \rangle}{\langle \alpha, \alpha \rangle} \in \bb{Z}$. Then the formula is given by
$$U(\mf{g})/J(\lambda) = \sum_{w \in W_{\lambda}} \det(w) X(\lambda,w\lambda),$$
where $X(\mu,\nu) = K$-finite part of $Ind_B^G(e^{(\mu,\nu)} \otimes 1)$ is the \textbf{principal series representation} with character $(\mu,\nu) \in \mf{h}_{\bb{C}}$, the complexification of the maximal torus $\mf{h}$ in $\mf{g}$ (here we treat $G$ as a real Lie group). In particular, the $G \cong K_{\bb{C}}$-types of $X(\mu,\nu)$ is equal to $Ind_T^G(e^{\mu-\nu})$ (see Theorem 1.8 of \cite{BV 1985} for more details on the principal series representations).\\
\indent Now apply the above recipe for $\lambda_1 = (1,1/2,-3/2)$: With the above notations, $W_{\lambda_1}$ is isomorphic to $W(A_1 \times \tilde{A}_1)$, generated by the roots $\{ (0,1,-1), (2,-1,-1)\}$. Hence the character formula of $U(\mf{g})/J(\lambda_1)$ is given by
\begin{align*}
U(\mf{g})/J(\lambda_1) = & X((1,1/2,-3/2),(1,1/2,-3/2)) - X((1,1/2,-3/2),(1,-3/2,1/2))\\
&- X((1,1/2,-3/2),(-1,3/2,-1/2)) +X((1,1/2,-3/2),(-1,-1/2,3/2)).
\end{align*}
Upon restricting to $K_{\bb{C}}$, we have
$$U(\mf{g})/J(\lambda_1)|_{K_{\bb{C}}} \cong Ind_T^G(e^{(0,0,0)}) - Ind_T^G(e^{(0,2,-2)}) - Ind_T^G(e^{(2,-1,-1)}) + Ind_G^T(e^{(2,1,-3)}).$$
Again, by $W(G_2)$-symmetry of finite-dimensional irreducible $G$-modules, the above expression can be written in the form as in the Proposition. The calculations for $U(\mf{g})/J(\lambda_2)$ is identical to the one above. We omit the calculations here.
\end{proof}

\noindent \textit{Proof of Theorem \ref{thm:main}.}
By the result of Levasseur and Smith in \cite{LS},
$$U(\mf{g})/J(\frac{1}{2}(\omega_1 + \omega_2))|_{K_{\bb{C}}} \cong R(\mc{O}_8)$$
as $G$-modules. Therefore the first equation of Proposition \ref{prop:uni} gives
\begin{equation} \label{eq:RO}
R(\mc{O}_8) \cong Ind_T^G(0,0) - Ind_T^G(0,1) - Ind_T^G(2,0) + Ind_T^G(1,1)
\end{equation}
as virtual $G$-modules. By Theorem \ref{thm:discrepancy} and Lemma \ref{lem:cos}, we need to show that
\begin{align*}
U(\mf{g})/J(\frac{1}{2}(5\omega_1 - \omega_2))|_{K_{\bb{C}}} &= R(\mc{O}_8) - \bigoplus_n V_{(1,n)} \\
&= R(\mc{O}_8) - (Ind_T^G(1,0) - Ind_T^G(0,1) - Ind_T^G(2,0)\\
&\ \ \ \  + Ind_T^G(1,1) + Ind_T^G(0,2) - Ind_T^G(2,1)).
\end{align*}
This is readily seen to be true from the Equation (\ref{eq:RO}) and the second equation of Proposition \ref{prop:uni}. \qed

\section{Final Remarks}
In \cite{S2000}, Sommers gives some conjectures on the multiplicities of small representations of $R(\mc{O})$ for the exceptional groups. In particular, given that his conjecture is true, one can show the non-normality of some orbit closures.\\
\indent To describe more explicitly which orbits $\mc{O}$ are conjectured to have non-normal closures, recall that Lusztig in \cite{Lu2} partitioned all nilpotent orbits in $\mf{g}$ by \textbf{special pieces}, i.e. for all nilpotent orbit $\mc{O}'$, it must belong to exactly one of the special pieces
$$S_{\mc{O}} := \{\mc{O}' \subseteq \overline{\mc{O}} | \mc{O}' \nsubseteq \overline{\mc{O}}_{spec}\ \text{for any other special orbit } \mc{O}_{spec} \subsetneq \mc{O} \},$$
where $\mc{O}$ runs through all special orbits in $\mf{g}$.\\
\indent For each $\mc{O}' \in S_{\mc{O}}$, Lusztig assigned a Levi subgroup $H(\mc{O}',\mc{O})$ of the Lusztig's quotient $\overline{A}(\mc{O})$. For example, the largest orbit in the special piece $\mc{O} \in S_{\mc{O}}$ has $H(\mc{O},\mc{O}) = 1$, and the smallest orbit $\mc{O}'' \in S_{\mc{O}}$ has $H(\mc{O}'',\mc{O}) = \overline{A}(\mc{O})$.\\
\indent By the conjecture of Sommers, if $\mc{O}$ has non-abelian Lusztig's quotient, i.e. $\overline{A}(\mc{O}) = S_3, S_4$ or $S_5$, then all $\mc{O}' \in S_{\mc{O}}$ with $H(\mc{O}',\mc{O})$ not equal to $1$ or $\overline{A}(\mc{O})$ (that is, not equal to $\mc{O}$ or $\mc{O}''$) have non-normal closures.\\
\indent For example, in the case of $G_2$ we studied above, we have $\mc{O}_8 \in S_{\mc{O}_{10}}$ and $H(\mc{O}_8,\mc{O}_{10}) = S_2 \leq S_3 = \overline{A}(\mc{O}_{10})$. So $\overline{\mc{O}_8}$ is conjectured to have non-normal closure, which has been shown to be true.\\
\indent We would like to end our manuscript with the following:
\begin{conjecture}
Suppose $\mc{O}$ is a nilpotent orbit with $\overline{A}(\mc{O}) = S_3$, $S_4$ or $S_5$, and $\mc{O}' \in S_{\mc{O}}$ satisfies $H(\mc{O}',\mc{O}) \neq 1$, $\overline{A}(\mc{O})$. Then there exists two distinct completely prime primitive ideals $J(\lambda_1)$, $J(\lambda_2)$ such that
$$U(\mf{g})/J(\lambda_1)|_{K_{\bb{C}}} \cong R(\mc{O}),\ \ \ U(\mf{g})/J(\lambda_2)|_{K_{\bb{C}}} \cong R(\overline{\mc{O}}).$$
\end{conjecture}

\end{document}